\documentclass[reqno]{amsart}
\usepackage{amssymb}
\usepackage{graphicx}

\DeclareSymbolFont{AMSb}{U}{msb}{m}{n}
\DeclareSymbolFontAlphabet{\Bb}{AMSb}

\usepackage{bbm}

\usepackage{tikz}

\usetikzlibrary{intersections,calc,positioning}
\usetikzlibrary{shapes.geometric,fit}
\usetikzlibrary{decorations,arrows,decorations.pathmorphing,backgrounds,positioning,fit,petri,shadows}
\usetikzlibrary{3d}
\usepackage{yhmath}
\usepackage{mathdots}
\usepackage{MnSymbol}

\newtheorem{theorem}{Theorem}[section]

\newtheorem{proposition}[theorem]{Proposition}
\theoremstyle{remark}

\newtheorem{remark}[theorem]{Remark}
\newtheorem{definition}[theorem]{Definition}
\numberwithin{equation}{section}
\begin{document}

\title[Combinatorial relations among relations]
{Two examples of combinatorial relations among relations of $C_{n}\sp{(1)}$-standard modules for higher levels }
\author{Tomislav \v Siki\' c }
\address{Tomislav \v{S}iki\'{c}, University of Zagreb, Faculty of Electrical Engineering and Com- \mbox{\hskip 8.5em} puting, Unska 3, 10000 Zagreb,  Croatia}
\email{tomislav.sikic@fer.hr}

\subjclass[2000]{Primary 17B67; Secondary 17B69, 05A19.}

\begin{abstract}
The construction of relations among relations is one ingredient in the Groebner-like basis construction of the maximal ideal of the universal vertex operator algebra $V^k_{\mathfrak g}$ for affine Lie algebras. For affine Lie algebras of type $C_n^{(1)}$, such combinatorially parametrized relations among relations were constructed in earlier work for level $2$ standard modules \cite{PS3}, and for $C_2^{(1)}$-standard modules at higher levels \cite{S}. This article presents two further examples in which the same counting method can be carried out. The first treats $C_n^{(1)}$-standard modules at the fixed level $k=5$, with $n$ arbitrary. The second treats $C_3^{(1)}$-standard modules for arbitrary level $k$. In both cases the calculation compares the number of required relations among relations in a trapezoid of the array of negative root vectors with the corresponding representation-theoretic dimension.
\end{abstract}
\maketitle
\def\sq{{\lower.3ex\vbox{\hrule\hbox{\vrule height1.2ex depth1.2ex\kern2.4ex			\vrule}\hrule}\,}}
\section{Introduction}

Let $\mathfrak g=\mathfrak{sp}_{2n}(\mathbb C)$, with $n\geq 2$, and consider standard modules for the affine Lie algebra of type $C_n^{(1)}$.  In the vertex-operator approach, annihilating fields give relations; the coefficients of these fields are organized in the space $\bar R$ of relations considered below.  After an order on monomials is fixed, the leading terms of these relations identify the monomials whose reductions have to be controlled.  This is the framework used in earlier constructions of combinatorial bases from annihilating fields, and in the symplectic affine case it leads to a description of the relevant monomials in terms of colored partitions; see, for example, \cite{MP1,PS1,PS2}.

The present paper concerns the next layer of this reduction process.  A colored partition $\pi$ may contain two different embeddings $\rho_1\subset \pi$ and $\rho_2\subset \pi$ of leading terms of relations.  The relations sought here express the difference of the corresponding reductions in the form
\begin{equation}\label{reduction}
u(\rho_1\subset \pi)-u(\rho_2\subset \pi)
 = \sum_{\pi\prec \pi',\,\rho\subset \pi'} c_{\rho\subset \pi'}u(\rho\subset \pi') .
\end{equation}
Identities of this form are the relations among relations studied in this article.  They provide the combinatorial overlap data used in the Groebner-like basis construction of the maximal ideal of the universal vertex operator algebra $V^k_{\mathfrak g}$.

This work continues the constructions of \cite{PS3} and \cite{S}.  In \cite{PS3}, relations among relations were constructed for level $2$ standard $C_n^{(1)}$-modules, with $n$ arbitrary.  In \cite{S}, the corresponding construction was carried out for standard $C_2^{(1)}$-modules at arbitrary level.  Instead of treating the full two-parameter problem for all pairs $(n,k)$, the present paper gives two further examples showing that the same counting method can be applied beyond the previously treated cases: arbitrary rank at the fixed level $k=5$, and arbitrary level for the fixed affine type $C_3^{(1)}$. 

The common setting is the array of negative root vectors $[\bar B_{<0}]_1^{2n+1}$.  The leading terms of relations for level $k$ standard modules are represented by downward zig-zag lines of length $k+1$ in this array.  The overlap calculation is localized to a trapezoid $T$ consisting of three consecutive triangles.  For a colored partition $\pi$ supported in $T$, let
\[
\mathcal E(\pi)=\{\rho\in \ell\!\text{{\it t\,}}(\bar R)\mid \rho\subset \pi\}
\]
be the set of embedded leading terms of relations in $\pi$; then
\[
N(\pi)=\max\{\#\mathcal E(\pi)-1,0\}
\]
measures the number of relations among relations needed for the embeddings of leading terms in $\pi$.  Summing these numbers over all relevant supports gives the combinatorial quantity $N_T(k)$.  The representation-theoretic comparison quantity is $\dim_T Q_{k+2}$, and Proposition~\ref{trapezenka} gives the formula used in the two comparisons below. It is important to emphasize that the motivation of this research is to prove the existence of relations like (\ref{reduction}) for colored partitions $\pi\in\mathcal P^{l}(m)$ in the general situation, i.e. for any  $l\in[k+2,2n+1]_{\mathbb{Z}}$ and for all pairs $(n,k)$. It is expected that all other relations will follow from the examples given in this article for $l=k+2$.

The first example fixes $k=5$ and keeps $n$ arbitrary.  Since $\ell(\pi)=k+2=7$, the possible supports of $\pi$ in the trapezoid can be organized into the four families of support types used in the overlap count.  Section~4 evaluates the corresponding sums and verifies the equality $N_T(5)=\dim_T Q_7$.  This equality is then used, via the sufficiency argument from Section~3, to prove Theorem~\ref{the main theorem 1}: for every pair of embeddings of leading terms in a colored partition $\pi\in\mathcal P^{7}(m)$, with $m$ arbitrary, the desired relation among relations is obtained for standard $C_n^{(1)}$-modules of level $5$.

The second example fixes $n=3$ and lets $k$ vary.  The array $[\bar B_{<0}]_1^7$ again has seven rows, so the list of possible support types remains finite and comparable to the first example, while the level remains a parameter.  Section~5 evaluates the contributions of all relevant support types and compares the resulting polynomial $N_T(k)$ with $\dim_T Q_{k+2}$.  The same sufficiency argument then gives Theorem~\ref{the main theorem 2}, the corresponding relation among relations for standard $C_3^{(1)}$-modules at arbitrary level $k$.

The paper is organized as follows.  Section~2 recalls the root-vector array and the description of leading terms as zig-zag lines.  Section~3 sets up the counting of embeddings in a trapezoid and computes the formula for $\dim_T Q_{k+2}$.  Sections~4 and~5 carry out the two examples described above.

Some related recent results employing combinatorial approach by \cite{LW} include  \cite{P}, \cite{CMPP}, \cite{DK}, \cite{R}, \cite{PT}, \cite{KRTW}, \cite{Kan}.
\section{The leading terms of relations for $C_{n}\sp{(1)}$-standard modules in the array of negative root vectors}

We fix a simple Lie algebra $\mathfrak{g}$ of type $C_n$, $n\geq 2$, i.e.  $\mathfrak{g}=\mathfrak{sp}_{2n}(\mathbb C)$. For a given Cartan subalgebra $\mathfrak h$ and the corresponding
root system $\Delta$ we can write (as in \cite{B})
\begin{equation*}
	\Delta = \{\pm(\varepsilon_i\pm\varepsilon_j) \mid i,j=1,...,n\}\backslash\{0\}
\end{equation*}
with simple roots 
$\alpha_1= \varepsilon_1-\varepsilon_2$,  $\alpha_2=\varepsilon_2-\varepsilon_3$,  \dots  $\alpha_{n-1}=\varepsilon_{n-1}-\varepsilon_{n}$, $\alpha_n = 2\varepsilon_n$.
Then $\theta=2 \varepsilon_1$. For each $\alpha\in\Delta$ we choose a root vector $X_{\alpha}$ such that $[X_{\alpha},X_{-\alpha}]=\alpha^{\vee}$. For root vectors
$X_{\alpha}$ we shall use the following notation:
$$\begin{array}{ccc}
	X_{ij}\quad \text{or just}\quad ij &  \text{if}\   &\alpha =\varepsilon_i + \varepsilon_j\ , \ i\leq j\,,\\
	X_{\underline{i}\underline{j}}\quad \text{or just}\quad \underline{i}\underline{j} & \ \text{if}\  &\alpha =-\varepsilon_i - \varepsilon_j\ , \ i\geq j\,,\\
	X_{i\underline{j}}\quad \text{or just}\quad i \underline{j} & \ \text{if}\  &\alpha =\varepsilon_i - \varepsilon_j\ , \ i\neq j\,.\\
\end{array}
$$
With the previous notation $x_\theta=X_{11}$. We also write for $i=1, \dots, n$
$$
X_{i\underline{i}}=\alpha_i^{\vee}\ \text{or just}\ i\underline{i} \,.
$$
These vectors $X_{ab}$ form a basis $B$ of $\mathfrak g$ which we shall write in a triangular scheme. For example, for $n=3$ the basis $B$ is
\begin{center}
	\begin{tikzpicture} [scale=0.7]
		\node at (0,0) {$11$};\node at (2,0) {$22$};\node at (4,0) {$33$};
		\node at (6,0) {$\underline{3}\underline{3}$};\node at (8,0) {$\underline{2}\underline{2}$};\node at (10,0) {$\underline{1}\underline{1}$};
		\node at (1,1) {$12$};\node at (3,1) {$23$};\node at (5,1) {$3\underline{3}$};
		\node at (7,1) {$\underline{3}\underline{2}$};	\node at (9,1) {$\underline{2}\underline{1}$};
		\node at (2,2) {$13$};\node at (4,2) {$2\underline{3}$};\node at (6,2) {$3\underline{2}$};\node at (8,2) {$\underline{3}\underline{1}$};
		\node at (3,3) {$1\underline{3}$};\node at (5,3) {$2\underline{2}$};\node at (7,3) {$3\underline{1}$};
		\node at (4,4) {$1\underline{2}$};	\node at (6,4) {$2\underline{1}$};
		\node at (5,5) {$1\underline{1}$};
	\end{tikzpicture}
\end{center}	
\begin{center}
	Figure 1
\end{center}
In order to simplify counting of embeddings of leading terms of relations, we will use  the usual matrix notation for the basis $B$, i.e. we will use $i=1,\dots, 2n$ for the first index for rows and $j=1,\dots, 2n$ for the second index for columns/diagonals. 
\noindent
Moreover, we shall write $\bar{B}_{<0}=\coprod_{j>0}{B}\otimes t^{-j}$ in 
the  following scheme
\bigskip

\begin{center}
	\begin{tikzpicture} [scale=0.5]
		\draw (0,0) -- +(10,0) -- +(5,5) -- cycle;
		\node at (5,2) {$B\otimes t^{-1}$};
		\node at (0.8,0.3) {$11$};\node at (8.7,0.3) {$1,2n$};\node at (5,4) {$2n,1$};
		\draw (10.5,0.5) -- +(-5,5) -- +(5,5) -- cycle;
		\node at (11,3.5) {$B\otimes t^{-2}$};
		\node at (10.5,1.3) {$2,2n$};\node at (7,5) {$2n+1,1$};\node at (14,5) {$2n+1,2n$};
		\draw (11,0) -- +(10,0) -- +(5,5) -- cycle;
		\node at (16,2) {$B\otimes t^{-3}$};
		\node at (13,0.3) {$1,2n+1$};\node at (19.5,0.3) {$1,4n$};\node at (16,4) {$2n,2n+1$};
		\node[fill=black, circle, inner sep=1pt] at (20.5,3){};\node[fill=black, circle, inner sep=1pt] at (21,3){};\node[fill=black, circle, inner sep=1pt] at (21.5,3){};
	\end{tikzpicture}
	\end{center}	
\begin{center}
	Figure 2
\end{center}	
which we call {\it the array of negative root vectors of $C_{n}\sp{(1)}$ and denote by $[\bar{B}_{<0}]_1^{2n+1}$}.
\noindent
In the Figure 3  the array $[\bar{B}_{<0}]_1^{7}$ (i.e. $n=3$) of negative root vectors of $C_{3}\sp{(1)}$ is  presented with
\begin{center}
	\begin{tikzpicture} [scale=0.6] 
		\node at (0,0) {$11$};
		\node at (1,1) {$21$};
		\node at (2,2) {$31$};
		\node at (3,3) {$41$};
		\node at (4,4) {$51$};
		\node at (5,5) {$61$};
		\node at (6,6) {$71$};
		\node at (2,0) {$12$};
		\node at (3,1) {$22$};
		\node at (4,2) {$32$};
		\node at (5,3) {$42$};
		\node at (6,4) {$52$};
		\node at (7,5) {$62$};
		\node at (8,6) {$72$};
		\node at (4,0) {$13$};
		\node at (5,1) {$23$};
		\node at (6,2) {$33$};
		\node at (7,3) {$43$};
		\node at (8,4) {$53$};
		\node at (9,5) {$63$};
		\node at (10,6) {$73$};
		\node at (6,0) {$14$};
		\node at (7,1) {$24$};
		\node at (8,2) {$34$};
		\node at (9,3) {$44$};
		\node at (10,4) {$54$};
		\node at (11,5) {$64$};
		\node at (12,6) {$74$};
		\node at (8,0) {$15$};
		\node at (9,1) {$25$};
		\node at (10,2) {$35$};
		\node at (11,3) {$45$};
		\node at (12,4) {$55$};
		\node at (13,5) {$65$};
		\node at (14,6) {$75$};
		\node at (10,0) {$16$};
		\node at (11,1) {$26$};
		\node at (12,2) {$36$};
		\node at (13,3) {$46$};
		\node at (14,4) {$56$};
		\node at (15,5) {$66$};
		\node at (16,6) {$76$};
		\node at (12,0) {$17$};
		\node at (13,1) {$27$};
		\node at (14,2) {$37$};
		\node at (15,3) {$47$};
		\node at (16,4) {$57$};
		\node at (17,5) {$67$};
		\node at (14,0) {$18$};
		\node at (15,1) {$28$};
		\node at (16,2) {$38$};
		\node at (17,3) {$48$};
		\node at (18,4) {$49$};
		\node at (16,0) {$19$};
		\node at (17,1) {$29$};
		\node at (18,2) {$39$};
		\node at (19,3) {$49$};
		\node at (18,0) {$1\, 10$};
		\node at (19,1) {$2\, 10$};
		\node at (20,2) {$3\, 10$};
		\node at (20,0) {$1\, 11$};
		\node at (21,1) {$2\, 11$};
		\node at (22,0) {$1\, 12$};
		\draw [dotted] (0,0) -- +(10,0) -- +(5,5) -- cycle;
		\draw [dotted] (6,6) -- +(5,-5) -- +(10,0) -- cycle;
		\draw [dotted] (12,0) -- +(10,0) -- +(5,5) -- cycle;
		\node[fill=black, circle, inner sep=1pt] at (19.5,4){};\node[fill=black, circle, inner sep=1pt] at (20,4){};\node[fill=black, circle, inner sep=1pt] at (20.5,4){};
	\end{tikzpicture}\ .
\end{center}
\begin{center}
	Figure 3
\end{center}
To describe leading terms we need the following partial order on $[\bar{B}_{<0}]_1^{2n+1}$  
\begin{equation}\label{partial order}
	X_{i,j}\trianglelefteq X_{p,r} \quad \text{if}\quad i\in\{1,\dots , p\}\quad \text{and}\quad j\in\{r,p+r-i\}.	
\end{equation}
In other words, $b \trianglelefteq a$ if $b=X_{i,j}$ lies in the cone bellow the vertex $a=X_{p,r}$, as depicted on Figure 4 below:
\begin{center}
	\begin{tikzpicture} [scale=0.5]
		\draw[dashed] (0,0) -- +(10,0) -- +(5,5) -- cycle;
		\draw[dashed] (10.5,0.5) -- +(-5,5) -- +(5,5) -- cycle;
		\draw[dashed] (11,0) -- +(10,0) -- +(5,5) -- cycle;
		\draw (-0.8,-0.3) -- +(22.6,0) --+(16.5,6.2) --+(6.1,6.2) -- cycle;
		\node[fill=green, circle, inner sep=4pt] at (14,5) {};
		
		\draw[green] (14,5) -- (15,4) -- (14,3) -- (13,2)  -- (14,1);
		
		\node[fill=green, circle, inner sep=4pt] at (13,2) {};
		
		\node[fill=green, circle, inner sep=4pt] at (14,3) {};
		
		\node[fill=green, circle, inner sep=4pt] at (15,4) {};
		
		\node[fill=green, circle, inner sep=4pt] at (14,1) {};
		
		\node at (9,4) [fill=green, circle, draw, inner sep=4pt] (A)  {a};
		\node at (10,2) [fill=green, circle, draw, inner sep=3pt] (B)  {b};
		\draw (5.5,0) -- (A) -- (12.5,0);
		\node at (11.5,3.2) {$b \trianglelefteq a$};
	\end{tikzpicture}\\
	Figure 4
\end{center}
In Figure 4, three successive triangles of $[\bar{B}_{<0}]_1^{2n+1}$ are marked in the trapezoid $T$. As will be seen later, all further considerations and calculations will be made on just such a trapezoid T.
With above settings we  define leading terms $\rho$ of relations $r(\rho)\in\bar R$  as follows (for details see \cite{PS2} and \cite{PS3}).
\begin{definition}
	The monomial
	\begin{equation}\label{E:the leading terms of relations}
		\rho=a_1^{m_1}a_2^{m_2}\dots a_s^{m_s}, \quad m_1+m_2+\dots+m_s=k+1,
	\end{equation}
	over {\it downward zig-zag  line} of $s$ points in $[\bar{B}_{<0}]_1^{2n+1}$
	$$
	a_1\vartriangleright a_2\vartriangleright\dots \vartriangleright a_s, \quad 1\leq s\leq k+1 
	$$
	is {\it the leading term}  of the relation $r(\rho)\in\bar R$ for level $k$ standard modules of affine Lie algebra of the type $C_n\sp{(1)}$.
\end{definition}
\noindent Moreover, these are all leading terms of $\bar R$ in $\mathcal P_{<0}$  (for more details see Section II. in \cite{PS3}). Also, it is very important to emphasize that the position of the trapezoid $T$ in Figure 4 is in accord with Figure 2 only when the middle triangle is $B\otimes t^j$ for $j$ even, and for $j$ odd the figure should be flipped. However, in our arguments this will make no difference because the flipped zig-zag downward line is again a zig-zag downward line. Also in Figure 4 the green zig-zag line symbolicaly intrepret arbitrary leading term.
\noindent

\section{Number of  embeddings for leading terms  in colored partitions}
In this section we shall count the number of embeddings $\rho\subset\pi$ for colored partitions  $\text{supp\,}\pi $ placed in a trapezoid of three successive triangles, i.e. in the trapezoid as in Figure 4.
For a colored partition
\begin{equation}\label{E:colored partition}
	\pi=\prod\sb{a\in\bar B_{<0}} a^{\pi(a)}
\end{equation}
we have  $|\text{supp\,}\pi |\leq \ell(\pi)$. The basis $\bar B_{<0}$ is writen as the array of $2n+1$ rows and $a_1\vartriangleright\dots\vartriangleright a_r$ implies $r\leq 2n+1$.	Let $\ell(\pi)=k+2$ and assume that $\pi$ allows two embeddings of leading terms of relations for level $k$ standard modules (i.e. two different zig-zag lines of length $k+1$). Then $\text{supp\,}\pi $ is one of the following types (see \emph{Lema 3.1.} in \cite{PS3}):
\begin{itemize}
	\item[$(A_{s})$] $\text{supp\,}\pi =\{a_1, \dots, a_s\}$, $s\geq 2$, $a_1\vartriangleright\dots\vartriangleright a_s$.
	\item[$(B_{s\,\delta})$] $\text{supp\,}\pi =\{a_1, \dots, a_{s},b,c\}$, $s\geq 1$, $a_1\vartriangleright\dots\vartriangleright a_{s}$, $a_{s}\vartriangleright b$, $a_{s}\vartriangleright c$ and $b$ and $c$ are not comparable. We set $\delta$ to be $\vert$  if $b$ and $c$ are in the same row, and $\vert\vert $ otherwise.
	\item[$(C_{\delta\, s})$] $\text{supp\,}\pi =\{b,c,a_1, \dots, a_{s}\}$, $r\geq 1$, $a_1\vartriangleright\dots\vartriangleright a_{s}$, $b\vartriangleright a_1$, $c\vartriangleright a_1$ and $b$ and $c$ are not comparable. We set $\delta$ to be $\vert$  if $b$ and $c$ are in the same row, and $\vert\vert $ otherwise
	\item[$(D_{s\, \delta\, t})$] $\text{supp\,}\pi =\{a_1, \dots, a_s,b,c,d_1,\dots,d_t\}$,  $s,t\geq 1$, $a_1\vartriangleright\dots\vartriangleright a_s$, $a_s\vartriangleright b\vartriangleright d_1$, $a_s\vartriangleright c\vartriangleright d_1$, $d_1\vartriangleright\dots\vartriangleright d_t$, and $b$ and $c$ are not comparable. We set $\delta$ to be $\vert$  if $b$ and $c$ are in the same row, and $\vert\vert $ otherwise.
\end{itemize}
For the counting the “number of relations among relations needed” for embeddings $\rho\subset\pi$ on the trapezoid $T$ we need the following notation
\begin{eqnarray*}
	N(\pi)=\max\{\#\mathcal E(\pi)-1,0\} &\text{where}& \mathcal E(\pi)=\{\rho\in \ell\text{\!\it t\,}(\bar R)\mid \rho\subset\pi\}
\end{eqnarray*}
and
\begin{eqnarray*}
	\Sigma_T(\mathcal{X}) & = & \sum_{S\subset T,\ S\text{ is of the type }\mathcal{X}}1\\
	&&\nonumber\\
	N_T(\mathcal{X}) &=& \sum_{\pi, \, \text{supp\,}\pi\subset T, \,\text{supp\,}\pi\text{ is of the type }\mathcal{X}}N(\pi)
\end{eqnarray*}
where  $\mathcal{X}$ is one of  above types $\mathcal{X}=A,B,C,D$ for $\text{supp\,}\pi $. 
Due to \emph{Lemma 3.2.} and \emph{Lemma 3.3}, again from \cite{PS3}, we have: 
\begin{eqnarray}\label{N_{T}(X)}
	N_{T}(A_s) & = & (s-1){k+1\choose s-1}\,\Sigma_T(A_s)\nonumber\\
	&&\nonumber\\
	N_{T}(B_{s\,\delta}) & = & {k-1\choose s-1}\,\Sigma_{T}(B_{s\,\delta})\nonumber\\
	&&\\
	N_{T}(C_{\delta\, s}) & = & {k-1\choose s-1}\,\Sigma_{T}(C_{\delta\, s})\nonumber\\
	&&\nonumber\\
	N_{T}(D_{s \,\delta\, t}) & = & {k-1\choose s+t-1}\,\Sigma_{T}(D_{s \,\delta\, t})\ .\nonumber
\end{eqnarray}

\noindent
To prove main theorems in \cite{PS3} (see \emph{Theorem 4.1.}) and \cite{S} (see \emph{Theorem 5.1.}) for a fixed ordinary partition $p$ of length $k+2$ we need to count 
$$
\sum_{\text{sh\,} \pi=p, \, \pi\in\mathcal P^{k+2}(m)} N(\pi).  
$$
By Theorem 7.4 and Lemma 7.5 in \cite{PS1} for  $\pi\in\mathcal P^{k+2}(m)$ we have a space of relations for annihilating fields 
\begin{equation}\label{Q_k+2}
	Q_{k+2}(m)  =  U(\mathfrak g)q_{(k+1)\theta}(m)\oplus U(\mathfrak g)q_{(k+2)\theta}(m)\oplus U(\mathfrak g)q_{(k+2)\theta-\alpha\sp*}(m)
\end{equation}
where $\alpha\sp*=\alpha_1=\varepsilon_1 -\varepsilon_2$. Moreover, from above equation we can write
\begin{equation}\label{dim Q_k+2}
		dim Q_{k+2}(m)  =  dim\, L({(k+1)\theta})\plus dim\, L({(k+2)\theta})\plus dim\, L({(k+2)\theta-\alpha\sp*})\ .
\end{equation}
Then we have
\begin{equation}\label{E:9.1a}
	\sum_{\pi\in\mathcal P^{k+2}(m)} N(\pi)\geq\dim Q_{k+2}(m)\ .
\end{equation}
Note that $dim Q_{k+2}(m)$ does not depend on parameter $m$, so by abusing the notation, in the rest of the paper we will use the shorter notation $dim Q_{k+2}$. 

\noindent
By Theorem 9.2 in \cite{PS1} it was sufficient to prove that that in (\ref{E:9.1a}) the equality holds. In an effort to generalize the method developed as in the cases of mentioned main theorems in \cite{PS3}  and \cite{S} we need to calculate the formula  (\ref{E:9.1a}) through the whole trapezoid $T$, where $n$ and $k$ are arbitrary.
For three successive triangles we have corresponding notation 
\begin{equation}\label{E:9.1Ta}
	N_T(k)=\sum_{\pi, \, \ell(\pi)=k+2, \, \text{supp\,} \pi\subset T} N(\pi) 
\end{equation}
where $k$ is arbitrary level. Accordingly, we denote the right-hand side of the inequality by $\dim_T Q_{k+2}$.
For the trapazoid scheme of three successive triangles in Figure 2 all shapes $\text{sh\,}\pi$ that appear are listed by Young tableauxes from $m=-(k+2)$ to $m=-3(k+2)$.
Indeed, we can list the following cases
$$	
\begin{array}{clllclllll}
	m=&&&&m=&m=&&&&m=\\
	-(k+2)&&&&-2(k+2)&-2(k+2)&&&&-3(k+2)\\
	\sq &\sq\sq&\cdots&\sq\sq&\sq\sq&\sq\sq\sq&\sq\sq\sq&\cdots&\sq\sq\sq&\sq\sq\sq \\
	\sq &\sq &\cdots&\sq\sq&\sq\sq&\sq\sq&\sq\sq&\cdots&\sq\sq\sq&\sq\sq\sq \\
	\vdots&\vdots&\vdots&\vdots&\vdots&\vdots&\vdots&\cdots&\vdots&\vdots\\
	\sq &\sq &\cdots&\sq\sq&\sq\sq&\sq\sq&\sq\sq&\cdots&\sq\sq\sq&\sq\sq\sq \\
	\sq &\sq &\cdots&\sq\sq &\sq\sq&\sq\sq&\sq\sq&\cdots&\sq\sq\sq&\sq\sq\sq \\ 
	\sq &\sq &\cdots&\sq &\sq\sq&\sq&\sq\sq&\cdots&\sq\sq&\sq\sq\sq     
\end{array}    		
$$
Using above list of Young tableaux we define 
\begin{equation}\label{def T Q_{k+2}}
	\dim_T Q_{k+2} = \sum_{m=k+2}^{3(k+2)}\dim Q_{k+2}- 2\times\dim L((k+2)\theta)\ .
\end{equation}
It is important to emphasize that the proof of the equality
\begin{equation}\label{E:9.1T}
	N_T(k)=\dim_T Q_{k+2}
\end{equation}
for three successive triangles (i.e. for the trapezoid)  implies the the proof of the equality in (\ref{E:9.1a}). In next proposition we calculate  the right-hand side of the equation (\ref{E:9.1T}) i.e. we proof the explicit formula for (\ref{def T Q_{k+2}}). 
\begin{proposition}\label{trapezenka}
Let $T$ be a trapezoid consisting of three consecutive triangles of the array of negative root vectors  $[\bar{B}_{<0}]_1^{2n+1}$. For $n\geq 2$ and $k\geq 1$ we have
\begin{eqnarray}\label{Desna}
	dim_T Q_{k+2}	&= &\frac{2n(k+2)(2k+5)-(2n+2k+3)}{k+2}{2n+2k+2\choose 2n-1}\nonumber\\
	&= &\frac{2n(k+2)(2k+5)-(2n+2k+3)}{k+2}{2n+2k+2\choose 2k+3}\nonumber \ .
\end{eqnarray}
\end{proposition}
\begin{proof}
The Weyl dimension formula 
in the case of symplectic Lie algebra $\mathfrak{g}=\mathfrak{sp}_{2n}$\\ (with the corresponding $\rho=n\varepsilon_1+(n-1)\varepsilon_2+\cdots +2\varepsilon_{n-1}+\varepsilon_n$) gives
\begin{eqnarray}
	\dim L((k+1)\theta) &=& {2n+2k+1\choose 2n-1} ,\label{E:11.1a}\\
	\dim L((k+2)\theta) &=& {2n+2k+3\choose 2n-1},\label{E:11.2a}\\
	\dim L((k+2)\theta-\alpha^{\star}) &=& \frac{(2k+3)(2n+2k+3)}{2n-1}{2n+2k+1\choose 2n-3}.\label{E:11.3a}
\end{eqnarray}
By inserting the formulas  (\ref{E:11.1a}), (\ref{E:11.2a}) and (\ref{E:11.3a}) into the equality (\ref{dim Q_k+2}), we get
\begin{equation}\label{tri sume_1}
	\dim Q_{k+2} = 2n\,	{2n+2k+2\choose 2n-1} = 2n\,	{2n+2k+2\choose 2k+3} .
\end{equation}
Using above equation (\ref{tri sume_1}) for three successive triangles in Figure 2  we have
calculation for the first statement of the proposition
\begin{eqnarray}\label{sum P^4(n)}
	\dim_T Q_{k+2} &=&\sum_{m=k+2}^{3(k+2)}\dim Q_{k+2}- 2\times\dim L((k+2)\theta)\nonumber\\
	& = &(2k+5)\times \dim Q_{k+2} (m)- 2\times\dim L((k+2)\theta)\\
	&= &\frac{2n(k+2)(2k+5)-(2n+2k+3)}{k+2}{2n+2k+2\choose 2n-1}\nonumber\ .
\end{eqnarray}
Obviously, the second statement follows from
$$
{2n+2k+2\choose 2n-1}={2n+2k+2\choose 2k+3}\ .
$$
\end{proof}
\begin{remark}
	In the previous proposition, an explicit formula was given for calculating $\dim_T Q_{k+2}$ as a polynomial of two positive integer variables $(n,k)$. Therefore, denote the left-hand side of (\ref{E:9.1T}) by $N_T(n,k)$. The calculation for $N_T(n,k)$ it is not so simple as in right-hand side case. However, the calculation for $N_T(n,k)$ is possible, but if one value in pair $(n,k)$ is fixed.  For instance, in \cite{PS3}  level $k=2$ is fixed and  $n$  is arbitrary and  $N_T(n,2)$ is a polynomial of degree $8$ in variable $n$. In \cite{S} the equality in (\ref{E:9.1T}) is proven for all levels $k$ in the case when n=2 and $N_T(2,k)$ is polynomial of degree $4$ in variable $k$. In order to show that the mentioned method can be applied generally, in this article  two  interconnected cases would be presented. The first refers to the ordered pair $(n; k=5)$ (see \emph{Theorem \ref{the main theorem 1}}), and the second to $(n=3; k)$ (see \emph{Theorem \ref{the main theorem 2}}). It is interesting to note that in the first example $k+2 = 7$ and in the second one $2n+1=7$, i.e. in both cases $\text{supp\,}\pi \leq 7$. By choosing these two examples, we have limited the number of  $\text{supp\,}\pi $ types. It should be emphasized that the calculation problem from a numerical point of view is not trivial. The total number of supports $\text{supp\,}\pi $ automatically increases with increasing parameters $n$ and $k$. In this article and in \cite{PS3}, \cite{S}, the software package \emph{Mathematica}  was used because of its symbolic calculus capabilities. But in cases where the level $k$ is fixed and larger then $5$   \emph{Mathematica's} could not made symbolic calculation of $N_T(n,k)$ for arbitrary $n$.
\end{remark}


\section{Example 1 - Combinatorial relations among relations of standard $C_n^{(1)}$-modules for level $k=5$  }

In this section we prove the theorem of combinatorially parametrized relations among relations in the case $(n; k=5)$:
\begin{theorem}\label{the main theorem 1}
	For any two embeddings
	$\rho_1 \subset \pi$ and $\rho_2 \subset \pi$ in $\pi\in\mathcal P^{7}(m)$,
	where $\rho_1, \rho_2 \in\ell \!\text{{\it t\,}}(\bar{R})$ and $m$ is arbitrary, we
	have a  relation for $C_n^{(1)}$
	\begin{equation}\label{E:9.2}
		u(\rho_1 \subset \pi) - u(\rho_2 \subset \pi) 
		=\sum_{\pi\prec \pi', \ \rho\subset\pi'}c_{\rho\subset\pi'}\,u(\rho\subset\pi').
	\end{equation}
\end{theorem}
\begin{proof}
To prove the theorem we need to count the number $N_T(5)$ where 
$$
N_T (k)=\sum_{\pi, \, \ell(\pi)=7, \, \text{supp\,} \pi\subset T} N(\pi).  
$$
and check that the equality $	N_T(5)=\dim_T Q_{7}$ holds. From \emph{Proposition \ref{trapezenka}}  we have
\begin{equation}\label{dim_T Q_7}
	dim_T Q_{7} = \frac{208n-13}{7}{2n+12\choose 13}\ .
\end{equation}
Since $k=5$ (i.e. $k+2=7$)  list of appropriate $\text{supp\,}\pi$ is:
\begin{eqnarray}\label{(X)7}
	(A_s) & for  &  s=2,3,\cdots ,7\nonumber\\
	(B_{s\,\delta}) & for & s=1,2,\cdots ,5\nonumber\\
	(C_{\delta\, s}) & for& s=1,2,\cdots ,5\nonumber\\
	(D_{s \,\delta\, t}) & for &  s,t=1,2,\cdots ,4\ where\ s+t\leq 5\ .
\end{eqnarray}
\noindent
Using the software package \emph{Mathematica}, based on the \emph{Lemmas 3.4.-3.12.}  from the \cite{PS3}, we get list of $\Sigma_T(\mathcal{X})$ for appropriate $\text{supp\,}\pi$ in cases $\mathcal{X} = A,B,C$
\begin{align*}
	\Sigma_T(A_7) & =  \frac{1}{227026800}(32432400 n - 158821380 n^2 + 19338228 n^3 + 
		681539131 n^4 - 675319645 n^5 \nonumber\\
		&- 186651465 n^6 + 327836509 n^7 + 
		7351058 n^8 - 47147100 n^9 - 3323320 n^{10} + 2290288 n^{11}\nonumber\\
		& + 445536 n^{12} + 29120 n^{13} + 640 n^{14})\nonumber\\
	\Sigma_T(A_6)& =  \frac{1}{7484400}(-1247400 n + 6026310 n^2 + 558261 n^3 - 26900753 n^4 + 
		19866495 n^5\nonumber\\
		& + 11764775 n^6 - 8762292 n^7 - 2402004 n^8 + 
		784080 n^9 + 284240 n^{10} + 27456 n^{11} + 832 n^{12})\nonumber\\	
	\Sigma_T(A_5)& =  \frac{1}{113400}(22680 n - 108558 n^2 - 40995 n^3 + 498860 n^4 - 212625 n^5 - 
		271194 n^6 + 52920 n^7\nonumber\\
		& + 50640 n^8 + 7920 n^9 + 352 n^{10}) \nonumber\\
	\Sigma_T(A_4) & =  \frac{1}{840} (-210 n + 1009 n^2 + 798 n^3 - 4613 n^4 + 2296 n^6 + 672 n^7 + 
		48 n^8)\nonumber\\
	\Sigma_T(A_3)& =  \frac{1}{90} (30 n - 151 n^2 - 225 n^3 + 590 n^4 + 420 n^5 + 56 n^6)\nonumber\\	
	\Sigma_T(A_2)& =  \frac{1}{6} (-3 n + 19 n^2 + 60 n^3 + 20 n^4)\nonumber\\
	\Sigma_{T}(B_{5\,|})& =  \frac{1}{97297200}(-1584360 n + 7022574 n^2 + 9580077 n^3 - 39422383 n^4 -	15405533 n^5 + 50132797 n^6\nonumber\\
	 &+ 10672376 n^7 - 19697964 n^8 - 3885024 n^9 + 1905904 n^{10} + 620672 n^{11} + 59072 n^{12} + 1792 n^{13})  \nonumber\\
	\Sigma_T(B_{4\,|})& =  \frac{1}{1247400}(27810 n - 126027 n^2 - 216777 n^3 + 676445 n^4 + 
		528495 n^5 - 731346 n^6 - 453816 n^7\nonumber\\
		& + 163680 n^8 + 113520 n^9 + 	17248 n^{10} + 768 n^{11})\nonumber\\
	\Sigma_T(B_{3\,|})& =  \frac{1}{22680}(-738 n + 3537 n^2 + 7870 n^3 - 16569 n^4 - 23772 n^5 + 8568 n^6 + 
	16320 n^7 + 4464 n^8 + 320 n^9)\nonumber\\
	\Sigma_T(B_{2\,|})& =  \frac{1}{630}(33 n - 182 n^2 - 553 n^3 + 490 n^4 + 1652 n^5 + 952 n^6 + 128 n^7)  \nonumber\\
	\Sigma_T(B_{1\,|})& =  \frac{1}{30}(-3 n + 25 n^2 + 120 n^3 + 140 n^4 + 48 n^5)\nonumber
\end{align*}	
\begin{align*}
	\Sigma_T(B_{5\,||})& =  \frac{1}{340540200}(5964840 n - 26385966 n^2 - 25304643 n^3 + 
		148562232 n^4 - 2430428 n^5 - 191219028 n^6\nonumber\\
		& + 36727691 n^7 + 
		79242306 n^8 - 15699684 n^9 - 10786776 n^{10} + 679952 n^{11} + 
		585312 n^{12} + 62272 n^{13}\nonumber\\
		& + 1920 n^{14})\nonumber\\
	\Sigma_T(B_{4\,||})& =  \frac{1}{3742200}(-84240 n + 373806 n^2 + 453915 n^3 - 2093333 n^4 - 
		508695 n^5 + 2642123 n^6\nonumber\\
		& + 176220 n^7 - 1023924 n^8 - 55440 n^9 + 
		100496 n^{10} + 18240 n^{11} + 832 n^{12}) \nonumber\\
	\Sigma_T(B_{3\,||})& =  \frac{1}{56700}(1710 n - 7713 n^2 - 11965 n^3 + 41795 n^4 + 25095 n^5 - 46914 n^6\nonumber\\
		& -	19560 n^7 + 12480 n^8 + 4720 n^9 + 352 n^{10}) \nonumber\\
	\Sigma_T(B_{2\,||})& =  \frac{1}{1260}(-54 n + 255 n^2 + 518 n^3 - 1239 n^4 - 1456 n^5 + 840 n^6 + 992 n^7 + 	144 n^8)\nonumber\\
	\Sigma_T(B_{1\,||})& =  \frac{1}{45}(3 n - 16 n^2 - 45 n^3 + 50 n^4 + 132 n^5 + 56 n^6) 
\end{align*}	
\begin{align*}
	\Sigma_T(C_{5\,|})& =  \frac{1}{7484400}(-249480 n + 474318 n^2 + 1451817 n^3 - 2658205 n^4 - 
		2085787 n^5 + 3361699 n^6\nonumber\\
		& + 1157266 n^7 - 1302180 n^8 - 316536 n^9 + 120208 n^{10} + 42592 n^{11} + 4160 n^{12} + 128 n^{13})\nonumber\\
	\Sigma_T(C_{4\,|})& =  \frac{1}{113400}(5670 n - 8577 n^2 - 35991 n^3 + 45905 n^4 + 63125 n^5 - 49056 n^6\nonumber\\
	& - 	41748 n^7 + 10320 n^8 + 8880 n^9 + 1408 n^{10} + 64 n^{11}) \nonumber\\
	\Sigma_T(C_{3\,|})& =  \frac{1}{2520}(-210 n + 205 n^2 + 1424 n^3 - 917 n^4 - 2702 n^5 + 280 n^6 + 	1456 n^7 + 432 n^8 + 32 n^9)\nonumber\\
	\Sigma_T(C_{2\,|})& =  \frac{1}{90}(15 n - 2 n^2 - 101 n^3 - 20 n^4 + 160 n^5 + 112 n^6 + 16 n^7)\nonumber\\
	\Sigma_T(C_{1\,|})& =  \frac{1}{6}(-3 n - 5 n^2 + 10 n^3 + 20 n^4 + 8 n^5) \nonumber\\
	\Sigma_T(C_{5\,||})& =  \frac{1}{48648600}(1584360 n - 3704814 n^2 - 7867197 n^3 + 20835295 n^4 + 	5674097 n^5 - 26713401 n^6\nonumber\\
		& + 2189044 n^7 + 10953800 n^8 - 1654224 n^9 - 1446016 n^{10} + 65728 n^{11} + 74880 n^{12} + 8192 n^{13} +	256 n^{14}) \nonumber\\
	\Sigma_T(C_{4\,||})& =  \frac{1}{623700} (-27810 n + 56907 n^2 + 153417 n^3 - 318197 n^4 - 
		185515 n^5 + 399586 n^6\nonumber\\
		& + 70356 n^7 - 152856 n^8 - 13200 n^9 + 14432 n^{10} + 2752 n^{11}+ 128 n^{12})\nonumber\\
	\Sigma_T(C_{3\,||})& =  \frac{1}{11340}(738 n - 1233 n^2 - 4486 n^3 + 6677 n^4 + 7140 n^5 - 7476 n^6 - 
	4224 n^7 + 1968 n^8 + 832 n^9 + 64 n^{10}) \nonumber\\
	\Sigma_T(C_{2\,||})& =  \frac{1}{315}(-33 n + 38 n^2 + 217 n^3 - 182 n^4 - 392 n^5 + 112 n^6 + 208 n^7 + 
	32 n^8) \nonumber\\
	\Sigma_T(C_{1\,||})& =  \frac{1}{15}(3 n - n^2 - 20 n^3 + 32 n^5 + 16 n^6)\nonumber
\end{align*}
The symmetry argument for type $(D_{s \,\delta\, t})$ show that
\begin{equation}\label{D sim}
	\Sigma_T(D_{s \,\delta \, t}) = \Sigma_T(D_{s'\,\delta \, t'}) \ for \ s+t = s'+t'\ . 
\end{equation}
From above equation (\ref{D sim}) and last line in (\ref{(X)7}) we have 
\begin{eqnarray}\label{Isti na listi}	
	N_T(D_{1 \,||\, 4)}= N_T(D_{4 \,||\, 1}) &=&N_T(D_{2 \,||\, 3})=N_T(D_{3 \,||\, 2}),\nonumber\\
	N_T(D_{1 \,|\, 4})= N_T(D_{4 \,|\, 1}) &=&N_T(D_{2 \,|\, 3})=N_T(D_{3 \,|\, 2}),\nonumber\\
	N_T(D_{2 \,||\, 2})= N_T(D_{1 \,||\, 3})&=&N_T(D_{3 \,||\, 1}), \\
	N_T(D_{2 \,|\, 2})= N_T(D_{1 \,|\, 3})&=&N_T(D_{3 \,|\, 1}),\nonumber\\
	N_T(D_{1 \,||\, 2})&= &N_T(D_{2 \,||\, 1}),\nonumber\\
	N_T(D_{1 \,|\, 2})&=& N_T(D_{2 \,|\, 1})\nonumber
\end{eqnarray}
Finally for case $D$ we get 
\begin{align}
	\Sigma_{T}(D_{s|t})\mid_{s+t=5} & =  \frac{1}{194594400} (-475200 n + 8565804 n^2 - 12779832 n^3 - 32208891 n^4 + 	18792202 n^5 + 36296832 n^6\nonumber\\
		& - 3981406 n^7 - 14258673 n^8 - 2074644 n^9 + 1551264 n^{10} + 517088 n^{11} + 53664 n^{12} + 1792 n^{13})\nonumber\\
	\Sigma_{T}(D_{s|t})\mid_{s+t=4} & =  \frac{1}{2494800} (14940 n - 179388 n^2 + 130493 n^3 + 704715 n^4 - 
	7425 n^5 - 688149 n^6\nonumber\\
	& - 230736 n^7 + 147510 n^8 + 91960 n^9 + 	15312 n^{10} + 768 n^{11})	\nonumber\\
	\Sigma_{T}(D_{s|t})\mid_{s+t=3} & =  \frac{1}{22680}(-288 n + 2754 n^2 + 281 n^3 - 10458 n^4 - 6657 n^5 + 5796 n^6 + 	6504 n^7 + 1908 n^8 + 160 n^9)	\nonumber\\
	\Sigma_{T}(D_{1|1}) & =  \frac{1}{2520} (66 n - 567 n^2 - 686 n^3 + 1575 n^4 + 2884 n^5 + 1512 n^6 + 256 n^7)	\nonumber\\
	\Sigma_{T}(D_{s||t})\mid_{s+t=5} & =  \frac{1} {1362160800} (-75600 n - 44699868 n^2 + 121462068 n^3 + 
	141600095 n^4 - 228988760 n^5 \nonumber\\
	& - 155649494 n^6 + 133397264 n^7 + 	71941155 n^8 - 26966940 n^9 - 14142128 n^{10} + 1059968 n^{11}  \nonumber \\
	& + 946400 n^{12} + 112000 n^{13} + 3840 n^{14}) \nonumber\\
	\Sigma_{T}(D_{s||t})\mid_{s+t=4} & =  \frac{1}{3742200} (-8100 n + 201456 n^2 - 363627 n^3 - 724504 n^4 + 
		584925 n^5 + 805123 n^6 \nonumber\\
		& - 211266 n^7 - 321387 n^8 - 9900 n^9 + 38896 n^{10} + 7968 n^{11} + 416 n^{12}) \nonumber\\
	\Sigma_{T}(D_{s||t})\mid_{s+t=3} & =  \frac{1}{56700} (360 n - 5004 n^2 + 4970 n^3 + 19165 n^4 - 3570 n^5 - 19047 n^6 - 3720 n^7 + 4710 n^8 \nonumber\\
	& + 1960 n^9 + 176 n^{10}) \nonumber\\
	\Sigma_{T}(D_{1||1}) & =  \frac{1}{1260} (-18 n + 187 n^2 - 28 n^3 - 707 n^4 - 322 n^5 + 448 n^6 + 368 n^7 + 	72 n^8) \nonumber
\end{align}
\noindent
Using (\ref{N_{T}(X)})  it follows
\begin{eqnarray}\label{N_{T}(X)7}
	N_{T}(A_s) & = & (s-1){k+1\choose s-1}\, \Sigma_T(A_s)\ for\ s=2,3,\cdots ,7\nonumber\\
	&&\nonumber\\
	N_{T}(B_{s\,\delta}) & = & {k-1\choose s-1}\, \Sigma_{T}(B_{s\,\delta})\ for\ s=1,2,\cdots ,5\nonumber\\
	&&\\
	N_{T}(C_{\delta\, s}) & = & {k-1\choose s-1}\,\Sigma_{T}(C_{\delta\, s})\ for\ s=1,2,\cdots ,5\nonumber\\
	&&\nonumber\\
	N_{T}(D_{s \,\delta\, t}) & = & {k-1\choose s+t-1}\,\Sigma_{T}(D_{s \,\delta\, t})\ for\ s,t=1,2,\cdots ,4\ where\ s+t\leq 5\nonumber\ .
\end{eqnarray}
and again using \emph{Mathematica}, from (\ref{N_{T}(X)7}), (\ref{Isti na listi}) and above list of polinomials $\Sigma_{T}(\mathcal{X})$, we have calculation
$$
\begin{array}{lll}
	N_T(5) &=&	\sum_{r=2}^{7}N_T(A_r) + \sum_{r=1}^{5}\sum_{\delta=\,|,\, ||}
	\left(N_T(B_{r\,\delta})+N_T(C_{\delta \,r})\right)+\\
	&&\\ 
	&&+\sum_{\delta=\,|,\, ||}\left(4\times N_T(D_{1 \,\delta\, 4}) +3\times N_T(D_{1 \,\delta\, 3})+2\times N_T(D_{1 \,\delta\, 2})+N_T(D_{1 \,\delta\, 1})\right)\\
	\\
	&=&\frac{n}{26195400}  (-7484400 + 73299060 n + 622498968 n^2 + 
	1754807691 n^3 + 2663728067 n^4 \\
	\\
	&& + 2525792269 n^5 +	1603570969 n^6 + 705990428 n^7 + 218654436 n^8 + 47519472 n^9 \\
	\\
	&& + 7095088 n^{10} + 693056 n^{11} + 39872 n^{12} + 1024 n^{13})\\
	\\
	& = &\frac{16n-1}{2^7\cdot 26195400}2n (2n+1) (2n+2) (2n+3) (2n+4) (2n+5) (2n+6) (2n+7) (2n+8)\\
	\\
	&& \cdot (2n+9)(2n+10) (2n+11)(2n+12)\\
	\\
	& = &\frac{13!}{2^7\cdot 26195400}{2n+12\choose 13}(16n-1)\ .
\end{array} 
$$
\noindent
Since
$$
\frac{13!}{2^7\cdot 26195400} =  \frac{1664}{7\cdot 128} = \frac{13}{7}\ ,
$$
finally we have
\begin{equation}\label{N_T(5)}
N_T(5) = \frac{208n-13}{7}\binom{2n+12}{13}\ .
\end{equation}
Now from (\ref{dim_T Q_7}) and (\ref{N_T(5)}) the statement of theorem follows.
\end{proof}


\section{Example 2 - Combinatorial relations among relations of standard $C_3^{(1)}$-modules for all levels}
\noindent
In the previous section, the list of possible supports was determined with a limit of $7=k+2$. In this case the upper limit is again $7$, but $7= 2n+1$ for that reason those two list is almost the same. Furthermore in the case  $[\bar{B}_{<0}]_1^7$ (i.e. $n=3$) we can state the following proposition.
\begin{proposition}\label{L: klasifikacija dva ulaganja}
	Let $\ell(\pi)=k+2$ and let trapezoid $T\subset [\bar{B}_{<0}]_1^{7}$ oriented as in Figure 4. Assume that $\pi$ allows two embeddings of leading terms of relations for level $k$ standard modules. Then $\text{supp\,}\pi $ is one of the following types :
\begin{equation}
	\begin{array}{rcll}
		(A_{s}) &for& s=2,3,4,5,6,7 & \\
		(B_{s\,|}) &for& s=1,2,3,4,5,6& \\
		(B_{s\,||}) &for& s=1,2,3,4,5& \\
		(C_{|\, s}) &for& s=1,2,3,4,5,6& \\
		(C_{||\, s}) &for& s=1,2,3,4,5&  \\
		(D_{s\, |\, t}) &for& s,t=1,2,3,4,5& where\ s+t\leq 6\\
		(D_{s\, ||\, t}) &for& s,t=1,2,3,4 &where\ s+t\leq 5 .
	\end{array}
\end{equation}
For the above listed types and for an arbitrary $k$  we have respectively
\begin{equation}\label{N_T_3}
	\begin{array}{ll}
	N_T(A_7) = 6\cdot 384\cdot \binom{k+1}{6}& N_T(A_6) = 5\cdot 1856\cdot\binom{k+1}{5} \\
	N_T(A_5) = 4\cdot 3648\cdot\binom{k+1}{4}&N_T(A_4) = 3\cdot 3708\cdot \binom{k+1}{3}\\
	 N_T(A_3) = 2\cdot 2037\cdot \binom{k+1}{2}& N_T(A_2) = 567\cdot \binom{k+1}{1}\\
	&\\
	N_T(B_{6\,|})=192\cdot\binom{k-1}{5}&\\
	N_T(B_{5\,|})=1072\cdot\binom{k-1}{4}&	N_T(B_{5\,||})=192\cdot\binom{k-1}{4}\\
	N_T(B_{4\,|})= 2496\cdot\binom{k-1}{3}& N_T(B_{4\,||})=1024\cdot\binom{k-1}{3}\\
	N_T(B_{3\,|})=3115\cdot\binom{k-1}{2}&	N_T(B_{3\,||})=2252\cdot\binom{k-1}{2}\\
	N_T(B_{2\,|})=2220 \cdot\binom{k-1}{1}& N_T(B_{2\,||})=2610\cdot\binom{k-1}{1}\\
	N_T(B_{1\,|})=882\cdot\binom{k-1}{0}&	N_T(B_{1\,||})=1680\cdot\binom{k-1}{0}\\
	&\\
	N_T(C_{6\,|})=160\cdot\binom{k-1}{5}&	\\
	N_T(C_{5\,|})=880\cdot\binom{k-1}{4}&	N_T(C_{5\,||})=160\cdot\binom{k-1}{4}\\
	N_T(C_{4\,|})= 2010\cdot\binom{k-1}{3}& N_T(C_{4\,||})=840\cdot\binom{k-1}{3}\\
	N_T(C_{3\,|})=2445\cdot\binom{k-1}{2}&	N_T(C_{3\,||})=1810\cdot\binom{k-1}{2}\\
	N_T(C_{2\,|})= 1680\cdot\binom{k-1}{1}& N_T(C_{2\,||})=2040\cdot\binom{k-1}{1}\\
	N_T(C_{1\,|})=630\cdot\binom{k-1}{0}&	N_T(C_{1\,||})=1260\cdot\binom{k-1}{0}\\
	&\\
	N_T(D_{1 \,|\, 5})= 96\cdot\binom{k-1}{5}&\\
	N_T(D_{5 \,|\, 1})= 96\cdot\binom{k-1}{5}&\\
	N_T(D_{2 \,|\, 4})= 96\cdot\binom{k-1}{5}&\\
	N_T(D_{4 \,|\, 2})= 96\cdot\binom{k-1}{5}&\\
	N_T(D_{3 \,|\, 3})= 96\cdot\binom{k-1}{5}&\\
	N_T(D_{1 \,|\, 4})= 512\cdot\binom{k-1}{4}&N_T(D_{1 \,||\, 4})=96\cdot\binom{k-1}{4}\\
	N_T(D_{4 \,|\, 1})= 512\cdot\binom{k-1}{4}&N_T(D_{4 \,||\, 1})=96\cdot\binom{k-1}{4}\\
	N_T(D_{2 \,|\, 3})= 512\cdot\binom{k-1}{4}&N_T(D_{2 \,||\, 3})=96\cdot\binom{k-1}{4}\\
	N_T(D_{3 \,|\, 2})= 512\cdot\binom{k-1}{4}&N_T(D_{3 \,||\, 2})=96\cdot\binom{k-1}{4}\\
	N_T(D_{2 \,|\, 2})= 1150\cdot\binom{k-1}{3}&N_T(D_{2 \,||\, 2})=488\cdot\binom{k-1}{3}\\
	N_T(D_{1 \,|\, 3})= 1150\cdot\binom{k-1}{3}&N_T(D_{1 \,||\, 3})=488\cdot\binom{k-1}{3}\\
	N_T(D_{3 \,|\, 1})= 1150\cdot\binom{k-1}{3}&N_T(D_{3 \,||\, 1})=488\cdot\binom{k-1}{3}\\
	N_T(D_{1 \,|\, 2})= 1397\cdot\binom{k-1}{2}&N_T(D_{1 \,||\, 2})=1034\cdot\binom{k-1}{2}\\
	N_T(D_{2 \,|\, 1})= 1397\cdot\binom{k-1}{2}&N_T(D_{2 \,||\, 1})=1034\cdot\binom{k-1}{2}\\
	N_T(D_{1 \,|\, 1})= 979\cdot\binom{k-1}{1}&N_T(D_{1 \,||\, 1})=1166\cdot\binom{k-1}{1}
	\end{array}
\end{equation}
\end{proposition} 
\begin{proof}
			
	Since the array $[\bar{B}_{<0}]_1^{7}$ consists of $7$ rows $\pi$ cannot be composed of more than seven elements of the base $B$ in the case of type $A_s$. But in the other three cases there are a differences depending on what the $\delta$ is. If $\delta =|$ then types $B_{6\,|}$, $C_{6\,|}$ and $D_{s|t}$ (where $s+t=6$)  are possible. All other support types, regardless of the choice of $\delta$, are the same as in list of previous \emph{Theorem \ref{the main theorem 1}}. i.e. $|\text{supp}\, \pi|\leq 7$. Thus the first statement is proven. The second statement follows from \emph{Lemmas 3.4 - 3.12} from \cite{PS3}. All calculation are made by \emph{Mathematica} for $n=3$.\\
	Therefore, the equalities (\ref{N_T_3}) follows from the list below.
\begin{equation}\label{Sigma_T_3}
	\begin{array}{ll}
		\Sigma_T(A_7) =  384& \Sigma_T(A_6) =  1856 \\
		\Sigma_T(A_5) =  3648&\Sigma_T(A_4) = 3708\\
		\Sigma_T(A_3) =  2037\cdot & \Sigma_T(A_2) = 567\\
		&\\
		\Sigma_T(B_{6\,|})=192&\\
		\Sigma_T(B_{5\,|})=1072&	\Sigma_T(B_{5\,||})=192\\
		\Sigma_T(B_{4\,|})= 2496& \Sigma_T(B_{4\,||})=1024\\
		\Sigma_T(B_{3\,|})=3115&	\Sigma_T(B_{3\,||})=2252\\
		\Sigma_T(B_{2\,|})=2220 & \Sigma_T(B_{2\,||})=2610\\
		\Sigma_T(B_{1\,|})=882&	\Sigma_T(B_{1\,||})=1680\\
		&\\
		\Sigma_T(C_{6\,|})=160&	\\
		\Sigma_T(C_{5\,|})=880&	\Sigma_T(C_{5\,||})=160\\
		\Sigma_T(C_{4\,|})= 2010& \Sigma_T(C_{4\,||})=840\\
		\Sigma_T(C_{3\,|})=2445&	\Sigma_T(C_{3\,||})=1810\\
		\Sigma_T(C_{2\,|})= 1680& \Sigma_T(C_{2\,||})=2040\\
		\Sigma_T(C_{1\,|})=630&	\Sigma_T(C_{1\,||})=1260\\
		&\\
		\Sigma_T(D_{1 \,|\, 5})= 96&\\
		\Sigma_T(D_{5 \,|\, 1})= 96&\\
		\Sigma_T(D_{2 \,|\, 4})= 96&\\
		\Sigma_T(D_{4 \,|\, 2})= 96&\\
		\Sigma_T(D_{3 \,|\, 3})= 96&\\
		\Sigma_T(D_{1 \,|\, 4})= 512&\Sigma_T(D_{1 \,||\, 4})=96\\
		\Sigma_T(D_{4 \,|\, 1})= 512&\Sigma_T(D_{4 \,||\, 1})=96\\
		\Sigma_T(D_{2 \,|\, 3})= 512&\Sigma_T(D_{2 \,||\, 3})=96\\
		\Sigma_T(D_{3 \,|\, 2})= 512&\Sigma_T(D_{3 \,||\, 2})=96\\
		\Sigma_T(D_{2 \,|\, 2})= 1150&\Sigma_T(D_{2 \,||\, 2})=488\\
		\Sigma_T(D_{1 \,|\, 3})= 1150&\Sigma_T(D_{1 \,||\, 3})=488\\
		\Sigma_T(D_{3 \,|\, 1})= 1150&\Sigma_T(D_{3 \,||\, 1})=488\\
		\Sigma_T(D_{1 \,|\, 2})= 1397&\Sigma_T(D_{1 \,||\, 2})=1034\\
		\Sigma_T(D_{2 \,|\, 1})= 1397&\Sigma_T(D_{2 \,||\, 1})=1034\\
		\Sigma_T(D_{1 \,|\, 1})= 979&\Sigma_T(D_{1 \,||\, 1})=1166
	\end{array}
\end{equation}  	
\end{proof}	
\noindent
Finally, the main result of this section is the following theorem:
\begin{theorem}\label{the main theorem 2}
	For any two embeddings
	$\rho_1 \subset \pi$ and $\rho_2 \subset \pi$ in $\pi\in\mathcal P^{k+2}(m)$,
	where $\rho_1, \rho_2 \in\ell \!\text{{\it t\,}}(\bar{R})$, we
	have a level $k$ relation for $C_3^{(1)}$
	\begin{equation}\label{E:9.2-C3}
		u(\rho_1 \subset \pi) - u(\rho_2 \subset \pi) 
		=\sum_{\pi\prec \pi', \ \rho\subset\pi'}c_{\rho\subset\pi'}\,u(\rho\subset\pi').
	\end{equation}
\end{theorem}
\begin{proof}
In order to verify this theorem  it is enough to prove that in (\ref{E:9.1a}) the equality holds.
	As we say before, it is  much easier to count the number $N(\pi)$ for all $\pi$ with support in a trapezoid $T$ of three successive triangles and check the equation (\ref{E:9.1T}).\\
	From (\ref{N_T_3}) it is obvious  that  
$$	
	N_T(k)=\sum_{\pi, \, \ell(\pi)=k+2, \, \text{supp\,} \pi\subset T} N(\pi) 
$$	
is a polynomial of degree $6$.
On the other side, from \emph{Proposition \ref{trapezenka}} we have
\begin{eqnarray}\label{dim Q_{n=3}}
		\dim_T Q_{k+2} &=& \frac{1}{5}(24k^2+104k+102)\binom{2k+8}{4}\\
		=&& \frac{1}{15}(21420 + 49023 k + 45557 k^2 + 22100 k^3 + 5920 k^4 + 832 k^5 + 48 k^6)\nonumber
\end{eqnarray}
i.e. the $\dim_T Q_{k+2}$ is also a polynomial of degree $6$. 
	\\
	Since both sides of the inequality (\ref{E:9.1T}) are polynomials of  degree $6$ it is sufficient to check that the equality holds  for  $7$ different values of $k$.\\	
	In \cite{PS1} the theorem is proved for basic modules i.e. in the case $k=1$. Moroever in the \cite{PS3}	the above theorem statement is proved for $k=2$. It is important to emphasize that both of the above results are proved for arbitrary $n$, i.e. for every symplectic affine Lie algebra $C_n^{(1)}$.\\
	Finally, to finish the proof of theorem it is enough to check equality in (\ref{E:9.1T})  for extra $5$ values of $k$.   For the uniformity of calculation, we will choose $k$ to be greater than $8$. When $k$ is greater than $8$, then the lists of support types are the same. Therefore, it stands 
	$$
		\begin{array}{lll}
			N_T(k) &=&	\sum_{r=2}^{7}N_T(A_r) + N_T(B_{6\,|}) +N_T(C_{6\,|})+
			\sum_{r=1}^{5}\sum_{\delta=\,|,\, ||}
			\left(N_T(B_{r\,\delta})+N_T(C_{\delta \,r})\right)+\\
			&&\\ 
			&&+N_T(D_{1 \,|\, 5})+N_T(D_{5 \,|\, 1})+N_T(D_{2 \,|\, 4})+N_T(D_{4 \,|\, 2})+N_T(D_{3 \,|\, 3})+\\
			\\
			&&+\sum_{\delta=\,|,\, ||}\left(4\times N_T(D_{1 \,\delta\, 4}) +3\times N_T(D_{1 \,\delta\, 3})+2\times N_T(D_{1 \,\delta\, 2})+ N_T(D_{1 \,\delta\, 1})\right)\\
		\end{array} 
	$$
for $k\geq 8$.  
Again we use \emph{Mathematica} for calculation and the calculation gives the following sequence	
	\begin{eqnarray}\label{N_{T}(3,4,5)}
		N_{T}(8) & = & 5\,249\,244\nonumber\\
		N_{T}(9) & = & 8\,916\,180\nonumber\\
		N_{T}(10) & = & 14\,504\,490\\
		N_{T}(11) & = & 22\,746\,150\nonumber\\
		N_{T}(12) & = & 34\,564\,752\nonumber\ .  
	\end{eqnarray}
From (\ref{dim Q_{n=3}}) for the right hand side of the equality (\ref{E:9.1T})  we have
	\begin{eqnarray}\label{tri sume}
		\dim_T Q_{10}(m) &=& 5\,249\,244\nonumber\\ 
		\dim_T Q_{11}(m) &=& 8\,916\,180\nonumber\\ 
		\dim_T Q_{12}(m) &=& 14\,504\,490\\
		\dim_T Q_{13}(m) &=& 22\,746\,150\nonumber\\
		\dim_T Q_{14}(m) &=& 34\,564\,752\nonumber
	\end{eqnarray}
Now, from (\ref{N_{T}(3,4,5)}) and (\ref{tri sume}) it is obvious that equality holds in  (\ref{E:9.1a})  i.e. the theorem is proven.
\end{proof}

\section*{Acknowledgement}
\noindent
This work is partially supported by the project  “Implementation of
cutting-edge research and its application as part of the Scientific
Center of Excellence for Quantum and Complex Systems, and
Representations of Lie Algebras“; Grant No. PK.1.1.10.0004,
co-financed by the European Union through the European Regional
Development Fund - Competitiveness and Cohesion Programme 2021-2027 and by the Croatian Science Foundation under the project IP-2022-10-9006.\\
The author of this paper is grateful to Mirko Primc for support and many stimulating discussions. 



\end{document}